\documentclass[intlimits]{amsart}
\usepackage{amssymb,latexsym,amsthm,amsmath}
\usepackage{hyperref}
\usepackage{commath}
\numberwithin{equation}{section}

\newcommand{\PnC}{{\mathbb{P}^n(\mathbb{C})}}
\newcommand{\PolymCn}{\mathcal{P}_m(\mathbb{C}^n)}
\newcommand{\bfB}{\mathbf{B}}
\newcommand{\fg}{\mathfrak{g}}
\newcommand{\mA}{\mathcal{A}}
\newcommand{\Z}{\mathbb{Z}}
\newcommand{\N}{\mathbb{N}}
\newcommand{\C}{\mathbb{C}}
\newcommand{\B}{\mathbb{B}}
\newcommand{\R}{\mathbb{R}}
\newcommand{\T}{\mathbb{T}}

\theoremstyle{theorem}
\newtheorem{thm}{Theorem}[section]
\newtheorem{corollary}[thm]{Corollary}

\theoremstyle{definition}
\newtheorem{defn}[thm]{Definition}

\theoremstyle{remark}
\newtheorem{remark}[thm]{Remark}

\allowdisplaybreaks

\title[Pseudo-homogeneous symbols and moment maps]{Toeplitz operators, pseudo-homogeneous symbols and moment maps on the complex projective space}
\author[Morales-Ramos et al.]{Miguel Antonio Morales-Ramos, Ra\'ul Quiroga-Barranco \\ and Armando S\'anchez-Nungaray}
\begin{document}

\thanks{Research supported by SNI and Conacyt Grants}

\maketitle

\begin{abstract}
	Following previous works for the unit ball (see \cite{Vasilevski-Pseudo}), we define quasi-radial pseudo-homogeneous symbols on the projective space and obtain the corresponding commutativity results for Toeplitz operators. A geometric interpretation of these symbols in terms of moment maps is developed. This leads us to the introduction of a new family of symbols, extended pseudo-homogeneous, that provide larger commutative Banach algebras generated by Toeplitz operators. This family of symbols provides new commutative Banach algebras generated by Toeplitz operators on the unit ball.
\end{abstract}

\section{Introduction}

The study of commutative algebras generated by Toeplitz operators has been a very interesting subject in the last decade. The first works studied commutative $C^*$-algebras generated by Toeplitz operators on the unit disk and the unit ball, where it was found an important connection with the geometry of the underlying space (see \cite{GQV-disk,QV-2007,QV-2008}). In particular, it was proved the commutativity of the $C^*$-algebras generated by Toeplitz operator with symbols invariant under the action of maximal Abelian groups. A dual study was carried out for the complex projective space $\PnC$ in \cite{QS-Projective}. It was proved that results similar to those of the unit ball hold for projective spaces, but in this case the symbols that yield commutative $C^*$-algebras are those that depend only on the radial part of the homogeneous coordinates.

On the other hand, Vasilevski \cite{Vasilevski2010} introduced a new family of commutative algebras generated by Toeplitz operators which are Banach algebras but not $C^*$-algebras. These Banach algebras are those generated by Toeplitz operators whose symbols are the so called quasi-radial quasi-homogeneous symbols on the unit ball. This was again brought into the setup of the projective space in \cite{QS-Quasi-Projective}. In this latter work, it was proved that the quasi-radial quasi-homogeneous symbols yield commutative Banach algebras generated by Toeplitz operators on each weighted Bergman space of the complex projective space. A remarkable fact found in \cite{QS-Quasi-Projective} is that it was also given a geometric description of the symbols by proving that the quasi-homogeneous symbols can be associated to an Abelian subgroup of holomorphic isometries and that one can also construct Lagrangian foliations over principal bundles over such Abelian subgroups. Furthermore, this was proved for both the projective space and the unit ball.

On the other hand, Garc\'ia and Vasilevski considered in \cite{Vasilevski2015} new families of commutative Banach algebras generated by Toeplitz operators over the Bergman space of the two-dimensional unit ball. These are obtained by a generalization of the quasi-homogeneous symbols over the two-dimensional unit ball. 

Recently Vasilevski \cite{Vasilevski-Pseudo} introduced two further and more general commutative Banach algebras over the unit ball. These algebras are generated by Toeplitz operator where the symbols are called quasi-radial pseudo-homogeneous symbols, with a variation to consider a product of spheres or a single sphere. Moreover, Vasilevski showed that the quasi-radial quasi-homogeneous symbols of both types have interesting properties. More precisely, if $a$ and $b$ symbols from the product of spheres type, then $T_{ab}=T_aT_b$. However, when the symbols $a$ and $b$ are from the case of a single sphere, then one still has $T_bT_a=T_aT_b$ but $T_{ab} \not=T_aT_b$.

Our goal is to extend the existence of such type of commutative Banach algebras generated by Toeplitz operators from the unit ball to the setup of the complex projective space. Thus we introduce the quasi-radial pseudo-homogeneous symbols for projective spaces, both the multi-sphere and single-sphere cases. Furthermore, we prove that the pseudo-homogeneous symbols have a geometric background which is supplied by the moment maps associated to Hamiltonian actions. The latter are given in our case by toral actions on projective spaces that restrict to the usual toral actions on complex spaces. In particular, our geometric interpretation for the pseudo-homogeneous symbols applies to the case of the unit ball as well. This allows us to raise some question about the existence of larger families of symbols that properly contain the pseudo-homogeneous ones. This motivates us to consider what we called extended pseudo-homogeneous symbols that embrace all the possibilities discussed so far for both the unit ball and the projective space.

The paper is organized as follows. After some preliminary remarks presented in Section~\ref{sec:geoprel}, in Sections~\ref{sec:qrph_symbols} and \ref{sec:singlesphere_symbols} we prove the existence of projective versions (dual to the unit ball) of commutative Banach algebras generated by Toeplitz operators whose symbols are quasi-radial pseudo-homogeneous, first for the multi-sphere case and then for the single-sphere case. This is achieved on each weighted Bergman space. In Section~\ref{sec:momentmaps} we consider some aspects of symplectic geometry as Hamiltonian actions, moment maps, symplectic reductions and  Delzant polytopes. Using these tools,  we give a geometric description of  quasi-radial pseudo-homogeneous symbols and we indicate how these techniques may be used to extend these families of symbols. Thus, we introduce in Section~\ref{sec:extended-symbols} the extended pseudo-homogeneous symbols, which contain many families contained up to this point. In particular, we exhibit commutative Banach algebras generated by Toeplitz operator with extended quasi-radial pseudo-homogeneous symbols. It is important to note that for this kind of symbols we have that  $T_{ab} \not=T_aT_b$ in general. Finally, we show in Section~\ref{sec:extended_unit_ball} how to obtain commutative Banach algebras in the case of the unit ball for our extended pseudo-homogeneous symbols.

\section{Geometry and analysis on $\PnC$}\label{sec:geoprel}
On the $n$-dimensional complex projective space $\PnC$ and for every $j = 1, \dots, n$ we consider on the open set
\[
	U_j=\{[w]\in\PnC:w_j\neq 0\}
\]
the homogeneous coordinates given by the holomorphic chart $\varphi_j:U_j\rightarrow \C^n$ where
\[
	\varphi_j([w])=\frac{1}{w_j}\left(w_0,\dots ,\widehat{w}_j,\dots ,w_n\right)=(z_1,\dots ,z_n).
\]
These yield the holomorphic atlas of $\PnC$.

The canonical K\"ahler structure on $\PnC$ is given by the closed $(1,1)$-form
\[
	\omega = i \partial\bar{\partial} \log f_j
\]
on $U_j$ for every $j = 0, \dots, n$, where 
\[
	f_j([w])=\sum_{k=0}^n\frac{w_k\bar{w}_k}{w_j\bar{w}_j}=1+\sum_{k=1}^nz_k\bar{z}_k.
\]
The corresponding Riemannian metric is the well known Fubini-Study metric. The volume element of $\PnC$ with respect to the Fubini-Study metric is defined by
\[
	\Omega=\frac{1}{(2\pi)^n}\omega^n.
\]

On the other hand, the hyperplane line bundle on $\PnC$ is given by:
\[
	H = \{ ([w],\lambda) \in \PnC\times\C^{n+1} : \lambda \in (\C w)^*\},
\]
which assigns to every point in $\PnC$ the dual of the line in $\C^{n+1}$ that such point represents. The bundle $H$ has a natural Hermitian metric $h$ inherited from the usual Hermitian inner product on $\C^{n+1}$. As usual, let us denote by $H^m$ the $m$-th tensor power of $H$ and by $h^{(m)}$ its corresponding Hermitian inner product. Hence, we denote by $L^2(\PnC, H^m )$ the $L^2$-completion of the space of smooth sections $\Gamma(\PnC, H^m)$ with respect to the Hermitian inner product:
\[
	\langle\zeta,\xi\rangle_m =\int_{\PnC} h^{(m)}(\zeta,\xi)\Omega,
\]
where $\zeta,\xi \in \Gamma(\PnC,H^m)$. Since $\PnC$ is compact, the space of global holomorphic sections of $H^m$, denoted by $\mA^2_m(\PnC)$, is finite dimensional and so closed in $L^2(\PnC,H^m)$. This yields the so called weighted Bergman space on $\PnC$ with weight $m$. 

We note that there is a trivialization of $H^m$ on $U_0$ given by
\begin{align*}
	\C^n \times \C &\rightarrow H^m|_{U_0} \\
	(z, c) &\mapsto ([1,z], \lambda(1,z) \mapsto c\lambda).
\end{align*}
This map together with $\varphi_0$ yield an isometry $L^2(\PnC,H^m) \rightarrow L^2(\C^n, \nu_m)$ where 
\[
	\nu_m = \frac{(n+m)!}{\pi^n m!}\frac{dV(z)}{(1+|z_1|^2 + \dots + |z_n|^2)^{n+m+1}},
\]
and $dV(z)$ denotes the Lebesgue measure on $\C^n$ (see \cite{QS-Projective} for further details). With respect to this isometry, the Bergman space $\mA^2_m(\PnC)$ is identified with the space of polynomials of degree at most $m$ on $\C^n$ which is denoted by $\PolymCn$. 

In the rest of this work we will use the multi-index notation without further notice. In particular, the set of monomials $z^\alpha$ on $z = (z_1, \dots, z_n) \in \C^n$ of degree at most $m$ is a basis for $\mA^2_m(\PnC)$. This basis is indexed by the set $J_n(m) = \{\alpha \in \N^n : |\alpha| \leq m \}$. And it is easy to compute that
\begin{equation}\label{eq:basis-Bergman}
	\bigg\{\bigg(\frac{m!}{\alpha!(m - |\alpha|)!}\bigg)^{\frac{1}{2}}z^\alpha: \alpha\in J_n(m) \bigg\}
\end{equation}
is an orthonormal basis of $\mA^2_m(\PnC)$.

On the other hand, the Bergman projection for $\mA^2_m(\PnC)$ is given by
\[
	(B_m\psi)(z)=\frac{(n+m)!}{\pi^n m!} \int_{\C^n} 
	\frac{\psi(w)(1+z_1\overline{w}_1\dots +z_n\overline{w}_n)^m}{(1+w_1\overline{w}_1+\dots +w_n\overline{w}_n)^{n+m+1}}  dV(w).
\]
Finally, the Toeplitz operator $T_a$ on $\mA^2_m(\PnC)$
with bounded symbol $a \in L^\infty(\PnC)$ is given by
\[
	T_a(f) = B_m(a f),
\]
for every $f \in \mA^2_m(\PnC)$.

\section{Quasi-radial pseudo-homogeneous symbols}\label{sec:qrph_symbols}
The symbols introduced in this section follow those considered in \cite{Vasilevski-Pseudo} as well as the notions developed in \cite{QS-Quasi-Projective}.

Let $k=(k_0,\dots,k_\ell) \in \N^{\ell+1}$ be a multi-index so that $|k| = n+1$. This partition provides a decomposition of the coordinates of $w \in \C^{n+1}$ as $w = (w_{(0)}, \dots, w_{(\ell)})$, where
\[
w_{(j)} = (w_{k_0+\dots+k_{j-1}+1}, \dots, w_{k_0+\dots+k_{j}})=(w_{j,1}, \dots, w_{j,k_{j}}),
\]
for every $j = 0, \dots, \ell$, and the empty sum is $0$ by convention. We further decompose $w \in \C^{n+1}$ as follows. For every $j = 0, \dots, \ell$ we define $r_{j} = |w_{(j)}|$. And for any such $j$ we write
\[
	w_{j,l} = r_j s_{j,l}t_{j,l},  
\]
where 
\[
	s_{j,l} = \frac{|w_{j,l}|}{r_j},\quad  t_{j,l} = \frac{w_{j,l}}{|w_{j,l}|}.
\]
Correspondingly, we write  $s_{(j)}=(s_{j,1}, \dots, s_{j,k_j})$ and $t_{(j)} = (t_{j,1}, \dots, t_{j,k_j}) \in \T^{k_j}$. We observe that for every $j$ we have $s_{(j)} \in S^{k_j - 1}_+ = S^{k_j - 1} \cap \R^{k_j}_+$, the subset of the real $(k_j-1)$-sphere with nonnegative coordinates.

\begin{defn}\label{def:pseudo-homogeneous-symbol}
A symbol $\psi \in L^\infty(\PnC)$ is called $k$-pseudo-homogeneous if it has the form
\[  
	\psi([w])=b(s_{(0)}, \dots, s_{(\ell)})t^p=b(s_{(0)}, \dots, s_{(\ell)})\prod_{j=0}^{\ell}t_{(j)}^{p_{(j)}},
\]
where $b(s_{(0)}, \dots, s_{(\ell ) })\in L^\infty(S_+^{k_0-1}\times\dots\times S_+^{k_\ell-1})$ and $p=(p_0,\dots ,p_n)\in\Z^{n+1}$ with $|p|=0$.
\end{defn}

A particularly useful case is that where $b(s_{(0)},\dots,s_{(\ell ) })=\prod_{j=0}^\ell b_j(s_{(j)})$ with $b_j\in L_\infty(S_+^{k_j-1})$, $j=0,\dots , \ell$.
 
It is easily seen from the definition of the variables involved that the value of a pseudo-homogeneous symbol is well defined in terms of homogeneous coordinates. In fact, we observe that the condition $|p| = 0$ is used to have a well defined expression in Definition~\ref{def:pseudo-homogeneous-symbol}. This can be checked, for example, following the arguments from Section~3 in \cite{QS-Quasi-Projective}. Furthermore, the proof of Lemma~3.6 from the previous reference ensures that we can take $k_0 = 1$ without any loss of generality. Hence, from now on we will consider partitions of the form $(1,k)$ where $k \in \N^\ell$ is a partition of $n$ and we will also call the corresponding symbols $k$-pseudo-homogeneous. With such restriction, the general form of a pseudo-homogeneous symbol is given by
\[  
	\psi([w])=b(s_{(1)},\dots ,s_{(\ell) })t^p=b(s_{(1)},\dots ,s_{(\ell) })\prod_{j=1}^{\ell}t_{(j)}^{p_{(j)}},
\]
since $s_{(0)} = t_{(0)} = 1$.

This convention allows us to consider a natural corresponding expression for the $k$-pseudo-homogeneous symbols on $\C^n$ through the canonical embedding $\C^n \hookrightarrow \PnC$ given by the chart $\varphi_0$. More precisely, the $k$-pseudo-homogeneous symbol $\psi$ associated to $p$ and $b$ on $\C^n$ is given by
\[
	z \mapsto \psi([1,z])=b(s_{(1)},\dots ,s_{(\ell)})\prod_{j=1}^\ell t_{(\ell)}^{p_{(j)}}, 
\]

Next, we consider the quasi-radial symbols on $\PnC$ as first introduced in \cite{QS-Projective}.

\begin{defn}\label{def:quasi-radial}
	Let $k=(k_0,\dots ,k_\ell)\in\N^{\ell+1}$ be a partition of $n+1$. A $k$-quasi-radial symbol is a function $a \in L^\infty(\PnC)$ that satisfies 
	\[
		a([w])=\tilde{a}(|w_{(0)}|,\dots , |w_{(\ell)}|)
	\]
	for some function $\tilde{a}:[0,+\infty)^{\ell+1}\rightarrow\C$ which is homogeneous of degree 0.
\end{defn}

A similar discussion as the one provided above applies in this case (see \cite{QS-Projective}). Hence, we will consider $k_0 = 1$ and $k$-quasi-radial symbols on $\C^n \hookrightarrow \PnC$ for $k \in \N^\ell$. Their corresponding expression is given by 
\[
	z \mapsto a([1,z]) = \widetilde{a}(|z_{(1)}|, \dots, |z_{(\ell)}|)
\]
defined for $z \in \C^n$ and $\widetilde{a} : \R^\ell \rightarrow \C$ some function.

Collecting the previous two types of symbols we obtain the following.

\begin{defn}
	Let $k \in \N^\ell$ be a partition of $n$. A $k$-quasi-radial $k$-pseudo-homogeneous symbol is a function $\PnC \rightarrow \C$ of the form $a\psi$ where $a$ is $k$-quasi-radial symbol and $\psi$ is a $k$-pseudo-homogeneous symbol.
\end{defn}

By the previous discussion, for a partition $k \in \N^\ell$ of $n$ we will consider symbols defined on $\C^n$ of the form
\begin{equation}\label{eq:qrph-symbol}
	\psi(z) = a(|z_{(1)}|, \dots, |z_{(\ell)}|) \prod_{j=1}^\ell b_j(s_{(j)})t_{(j)}^{p_{(j)}},
\end{equation}
where $p \in \Z^\ell$ satisfies $|p| = 0$.

It follows from Lemma~3.8 in \cite{QS-Quasi-Projective} that for a given $k$-quasi-radial symbol $a \in L^\infty(\C^n)$ the Toeplitz operator $T_a$ acting on $\mA^2_m(\PnC) \simeq \PolymCn$ satisfies
\[
	T_a(z^\alpha) = \gamma_{a,k,m}(\alpha) z^\alpha
\]
for every $\alpha \in \N^\ell$, where $\gamma_{a,k,m} : \N^\ell \rightarrow \C$ is given by
\begin{equation}\label{eq:gamma-qr}
	\begin{aligned}
		&\gamma_{a,k,m}(\alpha) = \\
		&=\frac{(n+m)!}{(m - |\alpha|)!\prod_{j=1}^\ell(k_j - 1 + |\alpha_{(j)}|)!}
			\int_{\R_+^\ell}\frac{a(\sqrt{r_1}, \dots, \sqrt{r_\ell})}{(1+r)^{n+m+1}}
			\prod_{j=1}^\ell r_j^{|\alpha_{(j)}| + k_j -1}\dif r_j
	\end{aligned}
\end{equation}

We now extend to quasi-radial pseudo-homogeneous symbols the previous result and the corresponding one for quasi-homogeneous symbols found in \cite{QS-Quasi-Projective}. Note that the following computation is dual to the one found in \cite{Vasilevski-Pseudo}. In what follows we will denote $b_j(\sqrt{s_{(j)}}) = b_j(\sqrt{s_{j,1}}, \dots, \sqrt{s_{j,k_j}})$. We will also consider the following sets
\begin{equation}\label{eq:BandDelta}
	\begin{aligned}
		\bfB_+^{k_j-1} &= \{ x \in \R_+^{k_j-1} \mid x_1^2 + \dots + x_{k_j - 1}^2 \leq 1 \} \\
		\Delta_{k_j - 1} &= \{ x \in \R_+^{k_j-1} \mid x_1 + \dots + x_{k_j - 1} \leq 1 \}.
	\end{aligned} 
\end{equation}

\begin{thm}\label{thm:Toeplitz_qrph}
	Let $\psi \in L^\infty(\C^n)$ be a $k$-quasi-radial $k$-pseudo-homogeneous symbol whose expression is given as in \eqref{eq:qrph-symbol}. Then, the Toeplitz operator $T_\psi$ acting on the weighted Bergman space $\mA^2_m(\PnC) \simeq \PolymCn$ satisfies
	\[
		T_\psi (z^\alpha) =
		\begin{cases}
			\gamma_{\psi,m}(\alpha) z^{\alpha + p} & \text{ if } \alpha + p \in J_n(m) \\
			0 & \text{ if } \alpha + p \notin J_n(m)
		\end{cases}
	\]
	for every $\alpha \in J_n(m)$, where
	\begin{align*}
		&\gamma_{\psi,m}(\alpha) =  \\
		&=\frac{(n+m)!}{(m-|\alpha|)!\prod_{j=1}^\ell(k_j-1+|\alpha_{(j)}|+\frac{1}{2}|p_{(j)}|)!} \\
		&\qquad\times
			\int_{\R^\ell_+} \frac{a(\sqrt{r_1},\dots ,\sqrt{r_\ell}) \prod_{j=1}^\ell r_j^{|\alpha_{(j)}|+\frac{1}{2}|p_{(j)}|+k_j-1}\dif r_j}{(1+r_1+\dots + r_n)^{n+m+1}} \\ 
		&\quad \times
			\prod_{j=1}^\ell \bigg(
			\frac{(k_j-1+|\alpha_{(j)}|+ \frac{1}{2}|p_{(j)}|)!}{\prod_{l=1}^{k_j}(\alpha_{j,l}+p_{j,l})!}  \\
		&\qquad \int_{\Delta_{k_j-1}} b_j(\sqrt{s_{(j)}}) 
			 \bigg(1-\sum_{l=1}^{k_j-1}s_{j,l}\bigg)^{\alpha_{j,k_j} + \frac{1}{2}p_{j,k_j}}
				\prod_{l=1}^{k_j-1} s_{j,l}^{\alpha_{j,l}+\frac{1}{2}p_{j,l}} \dif s_{j,l} \bigg).
	\end{align*}
\end{thm}
\begin{proof}
It is enough to compute the following inner product in the Bergman space $\mA^2_m(\PnC) \simeq \PolymCn$.
\begin{align*}
	&\langle T_\psi z^\alpha, z^\beta \rangle = \langle \psi z^\alpha,z^\beta \rangle = \\
		&= \frac{(n+m)!}{\pi^n m!}
			\int_{\C^n}\frac{\psi(z) z^\alpha \overline{z}^\beta}{(1+|z_1|^2+\dots + |z_n|^2)^{n+m+1}} \dif V(z)  \\
		&= \frac{(n+m)!}{\pi^n m!} 
			\int_{\R_+^n} \frac{a(r_1, \dots, r_\ell) 
				\prod_{j=1}^\ell b_j(s_{(j)}) \prod_{j=1}^n |z_j|^{\alpha_j + \beta_j + 1}}{(1+|z_1|^2+\dots + |z_n|^2)^{n+m+1}} \prod_{j=1}^n \dif |z_j| \\
		&\quad\times \prod_{u=1}^n\int_\T t_u^{\alpha_u + p_u - \beta_u}\frac{\dif t_u}{it_u}
\intertext{The last product of integrals is nonzero if and only if $\beta = \alpha + p$, and in this case it equals $(2\pi)^n$. Hence, we will assume that $\beta = \alpha + p$ for which we have}
		&= \frac{2^n(n+m)!}{m!}
			\int_{\R_+^n} \frac{a(r_1, \dots, r_\ell) 
				\prod_{j=1}^\ell b_j(s_{(j)}) 
				\prod_{j=1}^n |z_j|^{2\alpha_j + p_j + 1}}{(1 + |z_1|^2 + \dots + |z_n|^2)^{n+m+1}} 
				\prod_{j=1}^n\dif |z_j|
\intertext{If for every $j = 1, \dots, \ell$ we denote by $\dif S_j$ the volume element on $S^{k_j-1}$, apply spherical coordinates on $\R_+^{k_j}$ and use the identity $r_{j,l} = r_j s_{j,l}$, then we now have}
		&= \frac{2^n(n+m)!}{m!} 
		\int_{\R_+^\ell \times \prod_{j=1}^\ell S^{k_j-1}}
			\frac{a(r_1, \dots, r_\ell)}{(1+r_1^2 + \dots + r_\ell^2)^{n+m+1}} \\
		&\quad\times 
			 \prod_{j=1}^\ell b_j(s_{(j)})
			 \prod_{l=1}^{k_j} (r_j s_{j,l})^{2\alpha_{j,l} + p_{j,l} + 1}
			 \prod_{j=1}^\ell  r_j^{k_j-1} \dif r_j \dif S_j \\
		&= \frac{2^n(n+m)!}{m!}
			\int_{\R^\ell_+}\frac{a(r_1,\dots ,r_\ell) 
				\prod_{j=1}^\ell r_j^{2|\alpha_{(j)}|+|p_{(j)}|+2k_j-1} }{(1+r_1^2+\dots + r_\ell^2)^{n+m+1}}
				\prod_{j=1}^\ell \dif r_j \\
		&\quad\times 
			\prod_{j=1}^\ell  
			\int_{S^{k_j-1}} b_j(s_{(j)})\prod_{l=1}^{k_j} s_{j,l}^{2\alpha_{j,l} + p_{j,l} + 1} \dif S_j \\
		&= \frac{2^{n-\ell}(n+m)!}{ m!}
			\int_{\R^\ell_+}\frac{a(\sqrt{r_1},\dots ,\sqrt{r_\ell}) 
				\prod_{j=1}^\ell 
					r_j^{|\alpha_{(j)}|+\frac{1}{2}|p_{(j)}|+k_j-1}}{(1 + r_1 + \dots + r_\ell)^{n+m+1}}
					\prod_{j=1}^\ell \dif r_j \\ 
		&\quad\times 
			\prod_{j=1}^\ell  
			\int_{S^{k_j-1}} b_j(s_{(j)}) 
			\prod_{l=1}^{k_j} s_{j,l}^{2\alpha_{j,l}+p_{j,l}+1}  \dif S_j  \\
		&= \frac{(n+m)!}{m!}
			\int_{\R^\ell_+}\frac{a(\sqrt{r_1},\dots ,\sqrt{r_\ell}) 
				\prod_{j=1}^\ell 
					r_j^{|\alpha_{(j)}|+\frac{1}{2}|p_{(j)}|+k_j-1}}{(1 + r_1 + \dots + r_\ell)^{n+m+1}}
					\prod_{j=1}^\ell \dif r_j \\
		&\quad\times 
			\prod_{j=1}^\ell \bigg(2^{k_j-1} 
				\int_{S^{k_j-1}} b_j(s_{(j)})\prod_{l=1}^{k_j} s_{j,l}^{2\alpha_{j,l}+p_{j,l}+1}  
						\dif S_j \bigg)  \\
\intertext{Applying the identity $s_{j,k_j} = \sqrt{1-(s_{j,1}^2 + \dots + s_{j,k_j-1}^2)}$ together with the corresponding expression of $\dif S_j$ in terms of the coordinates $s_{j,1}, \dots, s_{j,k_j-1}$ on $\bfB_+^{k_j-1}$ we have}
		&= \frac{(n+m)!}{m!}
			\int_{\R^\ell_+} \frac{a(\sqrt{r_1}, \dots, \sqrt{r_\ell}) 
				\prod_{j=1}^\ell 
					r_j^{|\alpha_{(j)}|+\frac{1}{2}|p_{(j)}|+k_j-1}}{(1+r_1+\dots + r_\ell)^{n+m+1}}
					\prod_{j=1}^\ell \dif r_j  \\ 
		&\quad\times 
			\prod_{j=1}^\ell \bigg(2^{k_j-1} 
			\int_{\mathbf{B}_+^{k_j-1}} 
				b_j(s_{(j)}) \prod_{l=1}^{k_j-1} s_{j,l}^{2\alpha_{j,l}+p_{j,l}+1}
				\bigg(1- \sum_{l=1}^{k_j-1} s_{j,l}^2\bigg)^{\alpha_{j,k_j} + \frac{1}{2}p_{j,k_j}} 
					\prod_{l=1}^{k_j-1} \dif s_{j,l} \bigg) \\
		&= \frac{(n+m)!}{m!}
			\int_{\R^\ell_+} \frac{a(\sqrt{r_1}, \dots, \sqrt{r_\ell}) 
				\prod_{j=1}^\ell 
					r_j^{|\alpha_{(j)}|+\frac{1}{2}|p_{(j)}|+k_j-1}}{(1+r_1+\dots + r_\ell)^{n+m+1}}
					\prod_{j=1}^\ell \dif r_j  \\ 
		&\quad\times 
			\prod_{j=1}^\ell  
				\int_{\Delta_{k_j-1}} b_j(\sqrt{s_{(j)}})
					\prod_{l=1}^{k_j-1}  s_{j,l}^{\alpha_{j,l}+\frac{1}{2}p_{j,l}} 
						\bigg(1-\sum_{l=1}^{k_j-1} s_{j,l}\bigg)^{\alpha_{j,k_j}+\frac{1}{2}p_{j,k_j}} 
						\prod_{l=1}^{k_j-1} \dif s_{j,l}.
\end{align*}
On the other hand, we have
\[
	\langle z^{\alpha + p}, z^{\alpha + p}\rangle_m 
		= \frac{(\alpha+p)!(m-|\alpha|)!}{m!} 
		= \frac{(m-|\alpha|)!\prod_{j=1}^\ell\prod_{l=1}^{k_j}(\alpha_{j,l}+p_{j,l})!}{m!},
\]
because we assumed that $|p|=0$. The result now follows from this.
\end{proof}

The following is a particular case.

\begin{corollary}\label{cor:Toeplitz_qrph_pj0}
	For $\psi$ as above, we assume that $|p_{(j)}| = 0$ for every $j = 1, \dots, \ell$. Then, the function $\gamma_{\psi,m} : J_n(m) \rightarrow \C$ from Theorem~\ref{thm:Toeplitz_qrph} satisfies the following.
	\begin{align*}
		&\gamma_{\psi,m}(\alpha) =  \\
		 &= \frac{(n+m)!}{(m-|\alpha|)! \prod_{j=1}^\ell(k_j-1+|\alpha_{(j)}|)!}
			\int_{\R^\ell_+} \frac{a(\sqrt{r_1},\dots, \sqrt{r_\ell}) \prod_{j=1}^\ell r_j^{|\alpha_{(j)}|+k_j-1}\dif r_j}{(1+r_1+\dots + r_n)^{n+m+1}} \\ 
		 &\quad\times 
			\prod_{j=1}^\ell \Bigg(\frac{(k_j-1+|\alpha_{(j)}|)!}{\prod_{l=1}^{k_j}(\alpha_{j,l}+p_{j,l})!} \\ 
		 &\qquad
			 	\int_{\Delta_{k_j-1}} b_j(\sqrt{s_{(j)}})
			 	\bigg(1-\sum_{l=1}^{k_j-1}s_{j,l}\bigg)^{\alpha_{j,k_j} + \frac{1}{2}p_{j,k_j}}
			 	\prod_{l=1}^{k_j-1} s_{j,l}^{\alpha_{j,l} + \frac{1}{2}p_{j,l}} \dif s_{j,l}\Bigg),
	\end{align*}	
	for every $\alpha \in J_n(m)$ such that $\alpha + p \in J_n(m)$, and it is zero otherwise.
\end{corollary}

We observe that Corollary~\ref{cor:Toeplitz_qrph_pj0} reduces to formula \eqref{eq:gamma-qr} for quasi-radial symbols. Furthermore, for a symbol of the form 
\[
	\psi([w]) = b_j(s_{(j)}) t_{(j)}^{p_{(j)}},
\]
where $j$ is fixed and $p \in \Z^n$ satisfies $|p_{(j)}| = 0$ and $p_{(l)} = 0$ for every $l \not= j$, Corollary~\ref{cor:Toeplitz_qrph_pj0} gives the following expression 
\begin{align*}
	&\gamma_{b_j,k,p_{(j)}}(\alpha) = \\
		&= \frac{(k_j-1+|\alpha_{(j)}|)!}{\prod_{l=1}^{k_j}(\alpha_{j,l}+p_{j,l})!}
			\int_{\Delta_{k_j-1}} b_j(\sqrt{s_{(j)}}) 
			\bigg(1-\sum_{l=1}^{k_j-1}s_{j,l}\bigg)^{\alpha_{j,k_j} + 
			\frac{1}{2}p_{j,k_j}} \prod_{l=1}^{k_j-1} s_{j,l}^{\alpha_{j,l} + \frac{1}{2}p_{j,l}} \dif s_{j,l}
\end{align*}	
In particular, the Toeplitz operator with this latter symbol $\psi$ is independent of the weight $m$. 

As a consequence of the previous discussion we have the following result.

\begin{corollary}\label{cor:Toeplitz-commute}
	Let $\psi$ be a symbol as in Corollary~\ref{cor:Toeplitz_qrph_pj0}. Then, we have
	\[
		\gamma_\psi = \gamma_{a,k,m} \prod_{j=1}^\ell \gamma_{b_j,k,p_{(j)}}.
	\]
	Furthermore, the Toeplitz operators $T_a$, $T_{b_j t_{(j)}^{p_{(j)}}}$, for $j=1,\dots ,m$, pairwise commute and
	\[
		T_{a\prod_{j=1}^\ell b_j t_{(j)}^{p_{(j)}}}=T_a \prod_{j=1}^\ell T_{b_j t_{(j)}^{p_{(j)}}}.	
	\]
	In particular, the Toeplitz operators for these symbols generate a Banach algebra.
\end{corollary}

\section{Single-sphere pseudo-homogeneous symbols}\label{sec:singlesphere_symbols}
For a given $z \in \C^n$ we consider the decomposition
\[
	z_u = \rho_u t_u
\]
where $\rho_u = |z_u| \in \R_+$ and $t_u \in \T$, for every $u = 1, \dots, n$. We also introduce the coordinates
\[
	\sigma_u = \frac{\rho_u}{r}
\]
for every $u = 1, \dots, n$, where $r = |z|$. Hence, we have $\sigma = (\sigma_1, \dots, \sigma_n) \in S_+^{n-1}$. And with this notation we can write
\[
	z = r \sigma t = r(\sigma_1 t_1, \dots, \sigma_n t_n).
\]

\begin{defn}\label{def:single-pseudohomogeneous}
	For a partition $k \in \N^\ell$ of $n$, a symbol $\psi \in L^\infty(\PnC)$ is called single-sphere $k$-pseudo-homogeneous if it has the form
	\[
		\psi([w]) = b(\sigma) \prod_{j=1}^\ell t_{(j)}^{p_{(j)}}
	\]
	where $\sigma \in S_+^{n-1}$ and $t \in \T^n$ are computed for $z = \frac{(w_1, \dots, w_n)}{w_0}$, the function $b \in L^\infty(S_+^{n-1})$, and $p \in \Z^n$ satisfies $|p_{(j)}| = 0$ for every $j = 1, \dots, n$.
\end{defn}

We observe that with this definition the part $b(\sigma)$ of the symbols is well defined since it is given in terms of homogeneous coordinates. And as before, the condition on $p$ ensures that the part depending on $t$ is well defined as well. On the other hand, note that $\psi$ is given on the complement of a hyperplane of $\PnC$ and so it is well defined as an $L^\infty$ function. 

The difference between the single-sphere $k$-pseudo-homogeneous symbols and the $k$-pseudo-homogeneous symbols from Definition~\ref{def:pseudo-homogeneous-symbol} is that for the latter we partition the coordinates into several spheres according to the partition $k$. In particular, we can refer to the symbols from Definition~\ref{def:pseudo-homogeneous-symbol} as multi-sphere $k$-pseudo-homogeneous symbols.

As before, a particularly interesting case is the one for which we have
\[
	b(\sigma) = \prod_{j=1}^\ell b_j(\sigma_{(j)}),
\]
where $b_j \in L^\infty(S_+^{n-1} \cap \R^{k_j})$, where $\R^{k_j}$ is the subspace of $\R^n$ corresponding to the coordinates $(k_{j-1}+1, \dots, k_{j-1}+k_j)$. For simplicity, we will denote with $\chi_j$ this latter set.

We now fix a partition $k$ as above and we will now consider a symbol of the form
\begin{equation}\label{eq:j-single-sphere-pseudo-symbol}
	\psi_j(z) = b_j(\sigma_{(j)}) t_{(j)}^{p_j},
\end{equation}
where $p_j \in \Z^{k_j}$. As before, we can assume our symbol defined for $z \in \C^n$. Also note that for $\psi_j$ the multi-index $p_j$ needs to be given only on the coordinates in the set $\chi_j$, but it still has to satisfy $|p_j| = 0$. Since we need to consider multi-indexes for all the coordinates of $\C^n$, we introduce the following notation. Given $p_j \in \Z^{k_j}$, we will denote by $\widetilde{p}_j \in \Z^n$ the multi-index such that
\[
	(\widetilde{p}_j)_{(v)} =
	\begin{cases}
		p_j,  &\text{ if } v = j \\
		0, &\text{ if } v \not= j
	\end{cases}.
\]
In other words, $\widetilde{p}_j$ vanishes outside the coordinates in $\chi_j$ and coincides with $p_j$ in these coordinates.

\begin{thm}\label{thm:Toeplitz-j-singlesphere}
	Let $\psi_j \in L^\infty(\C^n)$ be a single-sphere $k$-pseudo-homogeneous symbol whose expression is given as in \eqref{eq:j-single-sphere-pseudo-symbol}. Then, the Toeplitz operator $T_{\psi_j}$ acting on the weighted Bergman space $\mA_m^2(\PnC) \simeq \PolymCn$ satisfies
	\[
		T_{\psi_j}(z^\alpha) =
		\begin{cases}
			\gamma_{\psi_j, m}(\alpha) z^{\alpha + \widetilde{p}_j}, &\text{ if }  
						\alpha + \widetilde{p}_j \in J_n(m) \\
			0, &\text{ if }  \alpha + \widetilde{p}_j \notin J_n(m)
		\end{cases},
	\]
	for every $\alpha \in J_n(m)$, where
	\begin{align*}
		&\gamma_{\psi_j, m}(\alpha) = \\
		&= \frac{(|\alpha|+n-1)!}{(|\alpha|-|\alpha_{(j)}|+n-k_j-1)!\prod_{u\in\chi_j}(\alpha_u+p_u)!} \\
		&\quad \times 
			\int_{\Delta_{k_j}} b_j(\sqrt{\sigma_{(j)}}) 
				(1 - (\sigma_{(j),1} + \dots + \sigma_{(j),k_j}))^{|\alpha|-|\alpha_{(j)}|+n-k_j -1} \\
		&\qquad \times
				\prod_{u \in \chi_j} \sigma_u^{\alpha_u + \frac{1}{2}(\widetilde{p}_j)_u} \dif \sigma_u.
	\end{align*}
\end{thm}
\begin{proof}
	As before we proceed to calculate as follows.
	\begin{align*}
		&\langle T_{\psi_j} z^\alpha, z^\beta \rangle_m = \langle \psi_j z^\alpha, z^\beta \rangle_m \\
		&= \frac{(n+m)!}{\pi^n m!}\int_{\C^n} 
			\frac{b_j(s_{(j)})t_{(j)}^{p_j} z^\alpha \overline{z}^\beta \dif V(z)}{(1+|z_1|^2+\dots + |z_n|^2)^{n+m+1}} \\
		&= \frac{(n+m)!}{\pi^n m!}
				\int_{\R^n_+} \frac{b_j(\sigma_{(j)})  \prod_{u=1}^n \rho_u^{\alpha_u+\beta_u+1}}{(1+\rho_1^2+\dots + \rho_n^2)^{n+m+1}} \dif\rho \\
		&\quad \times
				\prod_{u=1}^n \int_\T t_u^{\alpha_u + (\widetilde{p}_j)_u - \beta_u}\frac{dt_u}{it_u} \\
\intertext{The last product of integrals is nonzero precisely when $\beta = \alpha + \widetilde{p}_j$, in which case its value is $(2\pi)^n$. So we now assume that $\beta = \alpha + \widetilde{p}_j$}	
		&= \frac{2^n(n+m)!}{m!} \int_{\R^n_+} 
			\frac{b_j(\sigma_{(j)})  \prod_{u=1}^n \rho_u^{2\alpha_u+(\widetilde{p}_j)_u+1}}{(1 + \rho_1^2+ \dots + \rho_n^2)^{n+m+1}} \dif\rho \\
\intertext{We now replace $\rho_u = r\sigma_u$ for every $u = 1, \dots, n$ and apply spherical coordinates in $\R_+^n$ to obtain the following where $\dif \sigma$ is the volume element of $S_+^{n-1}$}
		&= \frac{2^n (n+m)!}{m!} \int_{\R_+ \times S_+^{n-1}} 
				\frac{b_j(\sigma_{(j)}) 
					\prod_{u=1}^n (r \sigma_u)^{2\alpha_u + (\widetilde{p}_j)_u + 1}}{(1 + r^2)^{n+m+1}}
					r^{n-1} \dif r \dif \sigma \\
		&= \frac{2^n (n+m)!}{m!} 
			\int_{\R_+} \frac{r^{2|\alpha| + 2n - 1}}{(1 + r^2)^{n+m+1}}\dif r 
			\int_{S_+^{n-1}} b_j(\sigma_{(j)}) 
				\prod_{u=1}^n \sigma_u^{2\alpha_u + (\widetilde{p}_j)_u + 1} \dif \sigma \\
		&= \frac{2^{n-1} (n+m)!}{m!} 
			\int_{\R_+} \frac{r^{|\alpha| + n - 1}}{(1 + r)^{n+m+1}}\dif r 
			\int_{S_+^{n-1}} b_j(\sigma_{(j)}) 
				\prod_{u=1}^n \sigma_u^{2\alpha_u + (\widetilde{p}_j)_u + 1} \dif \sigma \\
		&= \frac{(|\alpha|+n-1)!(m-|\alpha|)!}{m!}
			2^{n-1}\int_{S_+^{n-1}} b_j(\sigma_{(j)}) 
				\prod_{u=1}^n \sigma_u^{2\alpha_u + (\widetilde{p}_j)_u + 1} \dif \sigma \\
\intertext{Now we parametrize $S_+^{n-1}$ by $n-1$ coordinates as before. To simplify our computations we assume that $n \notin \chi_j$ and so that $(\widetilde{p}_j)_n = 0$}
		&= \frac{(|\alpha|+n-1)!(m-|\alpha|)!}{m!} \\
		&\quad \times
			2^{n-1}\int_{\bfB_+^{n-1}} b_j(\sigma_{(j)}) 
				\prod_{u=1}^{n-1} \sigma_u^{2\alpha_u + (\widetilde{p}_j)_u + 1} 
					(1 - (\sigma_1^2 + \dots + \sigma_{n-1}^2))^{\alpha_n} \prod_{u=1}^n \dif \sigma_u \\
		&= \frac{(|\alpha|+n-1)!(m-|\alpha|)!}{m!} \\
		&\quad \times
				\int_{\Delta_{n-1}} b_j(\sqrt{\sigma_{(j)}}) 
					\prod_{u=1}^{n-1} \sigma_u^{\alpha_u + \frac{1}{2}(\widetilde{p}_j)_u} 
					(1 - (\sigma_1 + \dots + \sigma_{n-1}))^{\alpha_n} \prod_{u=1}^n \dif \sigma_u \\
\intertext{We now apply the change of coordinates $\sigma_u = \tau_u$ for $u \in \chi_j$ and $\sigma_u = (1 - (\tau_{(j),1} + \dots + \tau_{(j), k_j}))\tau_u$ for $u \notin \chi_j$}
		&= \frac{(|\alpha|+n-1)!(m-|\alpha|)!}{m!} \\
		&\quad \times
			\int_{\Delta_{k_j}} b_j(\sqrt{\sigma_{(j)}}) 
				(1 - (\sigma_{(j),1} + \dots + \sigma_{(j),k_j}))^{|\alpha|-|\alpha_{(j)}|+n-k_j -1}
				\prod_{u \in \chi_j} \sigma_u^{\alpha_u + \frac{1}{2}(\widetilde{p}_j)_u} \dif \sigma_u \\
		&\quad \times
			\int_{\Delta_{n-k_j-1}} \bigg(1 - \sum_{u \notin \chi_j \cup \{n\}} \tau_u\bigg)^{\alpha_n}
					\prod_{u \notin \chi_j \cup \{n\}} \tau_u^{\alpha_u} \dif \tau_u \\
		&= \frac{(|\alpha|+n-1)!(m-|\alpha|)!}{m!} 		
			\frac{\prod_{u\notin\chi_j}\alpha_u!}{(|\alpha|-|\alpha_{(j)}|+n-k_j-1)!} \\
		&\quad \times
			\int_{\Delta_{k_j}} b_j(\sqrt{\sigma_{(j)}}) 
				(1 - (\sigma_{(j),1} + \dots + \sigma_{(j),k_j}))^{|\alpha|-|\alpha_{(j)}|+n-k_j -1} \\
		&\qquad \times		
				\prod_{u \in \chi_j} \sigma_u^{\alpha_u + \frac{1}{2}(\widetilde{p}_j)_u} \dif \sigma_u.		
	\end{align*}
	Since we can write
	\[
		\langle  z^{\alpha+\tilde{p}_{(j)}},z^{\alpha+\tilde{p}_{(j)}} \rangle_m
		= \frac{(m-|\alpha|)!\prod_{u\in\chi_j}(\alpha_u+p_u)!\prod_{u\notin\chi_j}\alpha_u!}{m!}
	\]
\end{proof}

Note that the action of the operator $T_{\psi_j}$ does not depend on the parameter $m$. From the previous results we also obtain the following.

\begin{corollary}\label{cor:quasi-radial-single-sphere-j-pseudo-homogeneous}
	For a partition $k \in \N^\ell$ of $n$, let $a \in L^\infty(\PnC)$ be $k$-quasi-radial symbol and for every $j = 1, \dots, \ell$ let $\psi_j \in L^\infty(\PnC)$ be a symbol of the form \eqref{eq:j-single-sphere-pseudo-symbol}. Then the Toeplitz operators $T_a$, $T_{\psi_j}$ pairwise commute on every weighted Bergman space $\mA^2_m(\PnC)$.
\end{corollary}

\begin{remark}\label{rmk:TabnotTaTb}
	It is easy to check that $T_{a\psi_j} \not= T_a T_{\psi_j}$ for suitably chosen symbols as above. This should be compared with Corollary~\ref{cor:Toeplitz-commute}.
\end{remark}

\section{Geometric description of symbols through moment maps}\label{sec:momentmaps}
After introducing some of the basic notions associated with Hamiltonian actions we will use some of their basic properties. We refer to \cite{Ortega-Ratiu} and \cite{Audin} for further details and proofs of known facts.

We will consider a compact group $G$ acting on a manifold $M$, which in our case will be compact as well. For simplicity and without loss of generality we will assume that the $G$-action is free: the stabilizer of every point is trivial. We also note that any action of a compact Lie group is necessarily proper.

For every $X \in \fg$ we will denote by $X^\sharp$ the vector field on $M$ whose flow is $\exp(tX)$. More precisely, we have
\[
	X^\sharp_x = \od{}{t}\sVert[2]_{t=0}\exp(tX)x.
\]
We further assume that the manifold carries a symplectic form $\omega$ and that the $G$-action preserves $\omega$. In particular, the $1$-form given by
\[
	\iota_{X^\sharp} \omega = \omega(X^\sharp, \cdot)
\]
is closed for every $X \in \fg$. The symplectic $G$-action on $(M,\omega)$ is called Hamiltonian if there exists a smooth map
\[
	\mu : M \rightarrow \fg^*,
\]
which is $G$-equivariant for the coadjoint $G$-action on $\fg^*$ (where $\fg^*$ denotes the dual vector space of the Lie algebra $\fg$) and if we have
\[
	\dif \mu_X = \iota_{X^\sharp} \omega
\]
for every $X \in \fg$. Here, the function $\mu_X \in C^\infty(M)$ is defined by  
\begin{align*}
	\mu_X : M &\rightarrow \R  \\
	\mu_X(x) &= \langle \mu(x), X \rangle.
\end{align*}
The map $\mu$ is called the moment map of the Hamiltonian $G$-action on $M$.

Marsden-Weinstein Theorem allows to perform what is known as the symplectic point reduction of a Hamiltonian action. More precisely, for the above setup, if $v \in \fg^*$ is a regular point of the moment map $\mu$, then the quotient
\[
	M_0 = \mu^{-1}(v)/G
\]
admits a manifold structure such that the quotient map
\[
	\pi : \mu^{-1}(v) \rightarrow M_0
\]
defines a principal bundle with structure group $G$. Furthermore, there exists a unique symplectic form $\omega_0$ on $M_0$ that satisfies
\[
	\pi^*(\omega_0) = \omega|_{\mu^{-1}(v)}.
\]
The symplectic manifold $(M_0, \omega_0)$ is called a symplectic point reduction of $(M,\omega)$.

This provides a well known alternative definition of the projective space that we now describe (see \cite{Audin} for further details). On $\C^{n+1}$ we consider the canonical symplectic form
\[
	\omega' = \frac{1}{2}\mathrm{Im}\bigg(\sum_{j=0}^n \dif w_j \wedge \dif \overline{w}_j\bigg),
\]
and the symplectic action
\begin{align*}
	\T \times \C^{n+1} &\rightarrow \C^{n+1} \\
		(t,w) &\mapsto (tw_0, \dots, tw_n).
\end{align*}
This $\T$-action is Hamiltonian with moment map
\begin{align*}
	\mu : \C^{n+1} &\rightarrow \R \\
		w &\mapsto \frac{1}{2}|w|^2.
\end{align*}
We recall that $\R$ is the Lie algebra of $\T$, and since the latter is Abelian the adjoint and coadjoint actions are trivial. The symplectic point reduction in this case, reduces to the well known fact that
\[
	\mu^{-1}\bigg(\frac{1}{2}\bigg) = S^{2n-1}/\T \simeq \PnC.
\]
It is also well known that the corresponding symplectic form is precisely the one introduced in Section~\ref{sec:geoprel}. 

Our second fundamental example is given by the action
\begin{equation}\label{eq:Tn-action}
	\begin{aligned}
		\T^n \times \PnC &\rightarrow \PnC \\
			(t, [w]) &\mapsto [tw] = [t_0 w_0, \dots, t_n w_n],
	\end{aligned}
\end{equation}
where we have identified
\[
	\T^n \simeq \{(t_0, \dots, t_n) \mid t_j \in \T \text{ for all } j = 0, \dots, n, \text{ and }
			t_0 \dots t_n = 1 \}.
\]
A straightforward computation shows that for the natural embedding $\C^n \hookrightarrow \PnC$ used before the subset $\C^n$ is $\T^n$-invariant with restricted action given by
\[
	t\cdot z = (t_1 z_1, \dots, t_n z_n),
\]
where $\T^n$ is now realized as the usual product of $n$ circles $\T$. This $\T^n$-action is Hamiltonian. In fact, if we realize the Lie algebra of $\T^n$ as
\[
	\widehat{\R}^n = \{ x \in \R^{n+1} \mid x_0 + \dots + x_n = \frac{1}{2} \},
\]
then the moment map of the $\T^n$-action on $\PnC$ is given by
\begin{align*}
	\mu_n : \PnC &\rightarrow \widehat{\R}^n  \\
		\mu_n([w]) &= \frac{1}{2|w|^2} (|w_0|^2, \dots, |w_n|^2).
\end{align*}
Note that $\widehat{\R}^n$ is an $n$-dimensional affine space, but it carries a natural structure of vector space by choosing the point
\[
	\frac{1}{2(n+1)}(1, \dots, 1)
\] 
as its origin.

This last example can be generalized to the notion of toric manifolds (see \cite{Audin} and \cite{Ortega-Ratiu}).

\begin{defn}
	A toric manifold is a compact symplectic manifold $M$ with dimension $2n$ together with an effective Hamiltonian $\T^n$-action. 
\end{defn}

A fundamental result due to Delzant establishes a one to one correspondence between toric manifolds and a certain family of polytopes.

\begin{thm}[Delzant]
	If $M$ is an $2n$-dimensional toric manifold with moment map $\mu$, then the image $\mu(M)$ is a polytope that satisfies the following properties.
	\begin{description}
		\item[Simple] There are $n$ edges meeting at each vertex.
		\item[Rational] The edges meeting at a vertex $p$ are of the form $p + tv_j$ for $t \in \R_+$, where $v_1, \dots, v_n \in \Z^n$.
		\item[Smooth] The vectors $v_1, \dots, v_n$ can be choose to be a basis for $\Z^n$. 
	\end{description}
	A convex polytope satisfying theses properties is called a Delzant polytope. Furthermore, there is a one to one correspondence
	\begin{align*}
		\{\text{toric manifolds} \} &\longrightarrow \{\text{Delzant polytopes}\} \\
			(M, \omega, \mu) &\longmapsto \mu(M).
	\end{align*}
\end{thm}

Delzant's Theorem applied to the above $\T^n$-action on $\PnC$ yields the following Delzant polytope given by the following expression and their corresponding realizations
\begin{align*}
	\mu_n(\PnC) &= \{ x \in \R^{n+1} \mid x_0 + \dots + x_n = \frac{1}{2} \} \\
			&\simeq \{ x \in \R^{n+1} \mid x_0 + \dots + x_n = 1 \} \\
			&\simeq \{ x \in \R_+^n \mid x_1 + \dots + x_n \leq 1 \},
\end{align*}
and the latter is precisely the set $\Delta_n$ introduced in Equation~\eqref{eq:BandDelta}. As a consequence, we obtain the following remark.

\begin{remark}\label{rmk:DelzantTeoplitzIntegrals}
	The expressions found in Sections~\ref{sec:qrph_symbols} and \ref{sec:singlesphere_symbols} that describe the action on monomials of the Toeplitz operators with pseudo-homogeneous symbols (both single and multi-sphere) are all given as integrals on the Delzant polytopes of projective spaces. Furthermore, the same remarks applies to the corresponding results for the unit ball: Lemmas~3.2 and 5.2 from \cite{Vasilevski-Pseudo}. These facts allow us to raise the following questions.
	\begin{itemize}
		\item Is there a relationship between the quasi-radial pseudo-homogeneous symbols and Delzant polytopes?
		\item With the introduction of the notion of Delzant polytopes, is it possible to infer the existence of further symbols whose associated Toeplitz operators can be explicitly described on monomials? 
	\end{itemize}
\end{remark}

We now proceed to answer the first question in Remark~\ref{rmk:DelzantTeoplitzIntegrals}. Hence, we fix a partition $k \in \Z^\ell$ of $n$ as before.

\subsection{Geometric description of $k$-quasi-radial symbols}\label{subsec:quasi-radial}
Let us consider the action \eqref{eq:Tn-action} restricted to the torus $\T^\ell$
\begin{equation}\label{eq:Tl-action}
	\begin{aligned}
		\T^\ell \times \PnC &\rightarrow \PnC \\
			(t,[w]) &\mapsto [t_0 w_0, t_1 w_{(1)}, \dots, t_l w_{(l)}].
	\end{aligned}
\end{equation}
As before we have made the identification 
\[
	\T^\ell \simeq \{(t_0, \dots, t_\ell) \mid t_j \in \T \text{ for all } j = 0, \dots, \ell, 
		\text{ and } t_0 \dots t_\ell = 1 \}.
\]
Since the previous action is the restriction of a Hamiltonian action it is Hamiltonian as well. Furthermore, the corresponding moment map is given by
\begin{align*}
	\mu_\ell : \PnC &\rightarrow \widehat{\R}^\ell \\
		[w] &\mapsto \frac{1}{2|w|^2} (|w_0|^2, |w_{(1)}|^2, \dots, |w_{(l)}|^2).
\end{align*}
This follows from the fact that the $\T^\ell$-action is a restriction of the action given by \eqref{eq:Tn-action}, for which the moment map is already known, and the results in Section~4.5 from \cite{Ortega-Ratiu}.

The following result describes the $k$-quasi-radial symbols as the pullback under the moment map $\mu_\ell$ of bounded functions on the polytope of such map.

\begin{thm}\label{thm:k-quasi-radial-polytope}
	A bounded symbol $a \in L^\infty(\PnC)$ is $k$-quasi-radial if and only if it belongs to $\mu_\ell^*(L^\infty(P_\ell))$, where $P_\ell = \mu_\ell(\PnC)$ the polytope associated to the moment map $\mu_\ell$. More precisely, $a$ is $k$-quasi-radial if and only if there exists $\widehat{a} \in L^\infty(P_\ell)$ such that $a = \widehat{a} \circ \mu_\ell$.
\end{thm}
\begin{proof}
	Suppose that $a = \widehat{a} \circ \mu_\ell$ for some $\widehat{a} \in L^\infty(P_\ell)$. Then, we can write
	\[
		a([w]) = \widehat{a}\bigg(\frac{|w_0|^2}{2|w|^2}, \frac{|w_{(1)}|^2}{2|w|^2} \dots, 	
			\frac{|w_{(\ell)}|^2}{2|w|^2}\bigg),
	\]
	which is clearly the composition of the map $w \mapsto (|w_0|, |w_{(1)}|, \dots, |w_{(\ell)}|)$ with a homogeneous function of degree $0$. Hence, $a$ is $k$-quasi-radial by Definition~\ref{def:quasi-radial}.
	
	Conversely, suppose that $a$ is a bounded $k$-quasi-radial symbol. By definition we can write
	\[
		a([w]) = \widetilde{a}(|w_0|, |w_{(1)}|, \dots, |w_{(\ell)}|)
	\]
	for some function $\widetilde{a}$ which is homogeneous of degree $0$. In particular, we have
	\[
		a([w]) = \widetilde{a}\bigg(\frac{|w_0|}{|w|}, \frac{|w_{(1)}|}{|w|}, \dots, \frac{|w_{(\ell)}|}{|w|}\bigg),
	\]
	and then we see that $a$ belongs to $\mu_\ell^*(L^\infty(P_\ell))$ by a trivial change of coordinates.
\end{proof}

\subsection{Geometric description of symbols of the form $t^p$}\label{subsec:tp-symbols}
Besides the partition $k$ we now fix $p \in \Z^n$ such that $|p_{(j)}| = 0$ for every $j = 1, \dots, \ell$. In particular, and according to the notation from Section~\ref{sec:qrph_symbols}, the symbol
\begin{align*}
	\PnC &\rightarrow \C \\
	[w] &\mapsto t^p = \prod_{j=1}^{\ell} t_{(j)}^{p_{(j)}},
\end{align*}
is $k$-pseudo-homogeneous.

Corresponding to the $\T^\ell$-action on $\PnC$ given by \eqref{eq:Tl-action} we have an induced quotient map
\begin{align*}
	\pi_k : \PnC &\rightarrow \prod_{j =1}^{k_j} \mathbb{P}^{k_j-1}(\C) \\
		[w] &\mapsto ([w_{(1)}], \dots, [w_{(\ell)}]).	
\end{align*}

On the other hand, the condition $|p_{(j)}| = 0$ for every $j = 1, \dots, \ell$ implies that the symbol $t^p$ is $\T^\ell$-invariant. As a consequence we now have the following.

\begin{thm}\label{thm:tp-type-pik}
	Every symbol of the type $t^p$ belongs to $\pi_k^*(L^\infty(\prod_{j =1}^{k_j} \mathbb{P}^{k_j-1}(\C)))$. More precisely, we have
	\begin{multline*}
		\{ t^p \mid p \in \Z^n, \text{ such that } |p_{(j)}| = 0 \text{ for every } j = 1, \dots, \ell \} \\
			\subset
				\bigg\{ f \circ \pi_k \mid f \in L^\infty\bigg(\prod_{j =1}^{k_j} \mathbb{P}^{k_j-1}(\C)\bigg) \bigg\}.	
	\end{multline*}
	In other words, any such symbol can be considered as a function on the product of projective spaces 
	$\prod_{j =1}^{k_j} \mathbb{P}^{k_j-1}(\C)$.
\end{thm}

\subsection{Geometric description of the pseudo-homogeneous symbols of the form $b(s_{(1)}, \dots, s_{(\ell)})$}
We will now consider a bounded function $b \in L^\infty(S_+^{k_1 -1}\times \dots \times S_+^{k_\ell -1})$ and the corresponding $k$-pseudo-homogeneous symbol
\begin{align*}
	\PnC &\rightarrow \C \\
		[w] &\mapsto b(s_{(1)}, \dots, s_{(\ell)}),
\end{align*}
where
\[
	s_{(j)} = \frac{1}{|w_{(j)}|} (|w_{j,1}|, \dots, |w_{j,k_j}|),
\]
for all $j = 1, \dots, \ell$.

From this description, it is easy to see that the symbol considered above is invariant under the $\T^n$-action given by \eqref{eq:Tn-action}. We now conclude the following result.

\begin{thm}\label{thm:b(s)-type-Delzant}
	Every symbol of the type $b(s_{(1)}, \dots, s_{(\ell)})$ belongs to $\mu_n^*(L^\infty(D_n))$ where $D_n = \mu_n(\PnC)$ is the Delzant polytope of the projective space $\PnC$. More precisely, we have
	\[
		\{ b(s_{(1)}, \dots, s_{(\ell)}) \mid b \in L^\infty(S_+^{k_1 -1} \times \dots \times S_+^{k_\ell -1}) \} 
		\subset	\{ f \circ \mu_n \mid f \in L^\infty(D_n) \}.	
	\]
	In other words, any such symbol can be considered as a function on the Delzant polytope of the projective space $\PnC$.
\end{thm}

\subsection{Geometric description of the single-sphere pseudo-homogeneous symbols}\label{subsec:single-sphere-symbols}
By the previous subsections, to describe these symbols it is enough to consider those of the form
\begin{align*}
	\PnC &\rightarrow \C \\
		[w] &\mapsto b(\sigma)
\end{align*}
for $b \in L^\infty(S_+^{n-1})$ and where we have followed the notation from Section~\ref{sec:singlesphere_symbols}. With the definition of the variables involved, these symbols are clearly invariant under the $\T^n$-action given by \eqref{eq:Tn-action}. Hence, we obtain the following result.

\begin{thm}\label{thm:b(sigma)-type-Delzant}
	Every symbol of the type $b(\sigma)$ belongs to $\mu_n^*(L^\infty(D_n))$ where $D_n = \mu_n(\PnC)$ is the Delzant polytope of the projective space $\PnC$. More precisely, we have
	\[
		\{ b(\sigma) \mid b \in L^\infty(S_+^{n-1}) \} 
		\subset	\{ f \circ \mu_n \mid f \in L^\infty(D_n) \}.	
	\]
	In other words, any such symbol can be considered as a function on the Delzant polytope of the projective space $\PnC$.
\end{thm}

\section{Extended pseudo-homogeneous symbols}\label{sec:extended-symbols}
In this section we bring upfront the second question from Remark~\ref{rmk:DelzantTeoplitzIntegrals}. 

In the geometric description of symbols considered in Section~\ref{sec:momentmaps} we have observed some proper inclusions in Theorems~\ref{thm:tp-type-pik}, \ref{thm:b(s)-type-Delzant} and \ref{thm:b(sigma)-type-Delzant}. This certainly suggests that there could be some room to consider additional symbols for which the description of their Toeplitz operators is possible and useful. In particular, further study of these symbols and their geometric description seems to be needed. We now present a new set of symbols that support these claims. As before, $k \in \N^\ell$ is a fixed partition of $n$.

\begin{defn}\label{def:extended-pseudo-homogeneous-symbol}
	A symbol $\psi \in L^\infty(\PnC)$ is called extended $k$-pseudo-homogeneous if it has the form
	\[  
		\psi([w]) = b(r_1, \dots, r_\ell, s_{(1)}, \dots ,s_{(\ell)}) t^p
		= b(r_1, \dots , r_\ell, s_{(1)}, \dots, s_{(\ell)}) \prod_{j=1}^\ell t_{(j)}^{p_{(j)}},
	\]
	where $b \in L^\infty([0,\infty)^\ell \times S_+^{k_1-1} \times \dots \times S_+^{k_\ell-1})$ and $p \in \Z^n$ satisfies $|p| = 0$.
\end{defn}

We observe that the set of quasi-radial pseudo-homogeneous symbols considered in the previous sections is a proper subset of the set of extended pseudo-homogeneous symbols. Furthermore, in the notation of Section~\ref{sec:singlesphere_symbols} we have
\[
	\sigma = \frac{1}{\sqrt{r_1^2 + \dots + r_\ell^2}} (r_1 s_{(1)}, \dots, r_\ell s_{(\ell)})
\]
from which it follows that every single-sphere quasi-radial pseudo-homogeneous symbol is also an extended pseudo-homogeneous symbol. In other words, all the symbols considered in previous sections are particular cases of the extended pseudo-homogeneous symbols.

A particular case that will be interesting to consider is given by symbols of the form
\begin{equation}\label{eq:extended-symbol}
	\psi([w])=a(r_1, \dots, r_\ell) \prod_{j=1}^\ell b_j(r_{1},\dots, r_{\ell}, s_{(j)}) t_{(j)}^{p_{(j)}},
\end{equation}
where we continue to assume conditions as in Definition~\ref{def:extended-pseudo-homogeneous-symbol}. Following the previous remarks, we observe that the symbols in \eqref{eq:extended-symbol} have as particular case those considered by the expressions \eqref{eq:qrph-symbol}, \eqref{eq:j-single-sphere-pseudo-symbol}.

\begin{thm}\label{thm:Toeplitz_extended}
	Let $\psi \in L^\infty(\C^n)$ be an extended $k$-pseudo-homogeneous symbol whose expression is given as in \eqref{eq:extended-symbol}. Then, the Toeplitz operator $T_\psi$ acting on the weighted Bergman space $\mA^2_m(\PnC) \simeq \PolymCn$ satisfies
	\[
		T_\psi(z^\alpha) =
			\begin{cases}
				\gamma_{\psi,m}(\alpha) z^{\alpha + p}, &\text{ if }  \alpha + p \in J_n(m)  \\
				0, &\text{ if }  \alpha + p \notin J_n(m)
			\end{cases},
	\]
	for every $\alpha \in J_n(m)$, where
	\begin{align*}
		&\gamma_{\psi,m}(\alpha) = \\
		&= \frac{(n+m)!}{(m-|\alpha|)! \prod_{j=1}^\ell(k_j-1+|\alpha_{(j)}|+\frac{1}{2}|p_{(j)}|)!}
			\int_{\R^\ell_+} \frac{a(\sqrt{r_1},\dots ,\sqrt{r_\ell})}{(1+r_1+\dots + r_\ell)^{n+m+1}}  \\
		&\quad\times
			\bigg(
				\int_{\prod_{j=1}^\ell \Delta_{k_j-1}} 
				\prod_{j=1}^\ell 
				\frac{(k_j-1+|\alpha_{(j)}|+\frac{1}{2}|p_{(j)}|)!}{\prod_{l=1}^{k_j}(\alpha_{j,l}+p_{j,l})!}  
					\bigg(1-\bigg(\sum_{l=1}^{k_j-1} s_{j,l}\bigg)\bigg)^{\alpha_{j,k_j}+\frac{1}{2}p_{j,k_j}}  \\ 
		&\qquad\times 
				b_j(\sqrt{r_{1}},\dots , \sqrt{r_{\ell}},\sqrt{s_{(j)}}) 		
					\prod_{l=1}^{k_j-1} s_{j,l}^{\alpha_{j,l}+\frac{1}{2}p_{j,l}}  \dif s_{j,l}
			\bigg) 
			\prod_{j=1}^\ell r_j^{|\alpha_{(j)}|+\frac{1}{2}|p_{(j)}|+k_j-1} \dif r_j.
	\end{align*}
\end{thm}
\begin{proof}
	We compute the inner product
	\begin{align*}
		&\langle T_\psi z^\alpha, \zeta^\beta \rangle = \langle \psi z^\alpha, z^\beta \rangle \\
		&= \frac{(n+m)!}{\pi^n m!}
			\int_{\C^n} 
			\frac{\psi(z) z^\alpha \overline{z}^\beta}{(1+|z_1|^2+\dots + |z_n|^2)^{n+m+1}}dV(z) \\
		&= \frac{(n+m)!}{\pi^n m!}
			\int_{\R_+^n} 
			\frac{a(r_1, \dots, r_\ell) \prod_{j=1}^\ell b_j(r_1, \dots, r_\ell, s_{(j)})
					}{(1+\rho_1^2+\dots + \rho_n^2)^{n+m+1}}
					\prod_{u=1}^n \rho_u^{\alpha_u + \beta_u + 1} \dif \rho_u \\
		&\quad\times
			\prod_{u=1}^n \int_\T t_u^{\alpha_u + p_u - \beta_u} \frac{\dif t_u}{it_u} \\
\intertext{This vanishes unless $\beta = \alpha + p$, and this case the computation continues as follows by introducing spherical coordinates in each variable $z_{(j)}$}
		&= \frac{2^n(n+m)!}{m!}
			\int_{\R_+^\ell \times \prod_{j=1}^\ell S^{k_j - 1}} 
			\frac{a(r_1, \dots, r_\ell)}{(1+ r_1^2 +\dots + r_n^2)^{n+m+1}}  \\
		&\quad\times	
				 \prod_{j=1}^\ell b_j(r_1, \dots, r_\ell, s_{(j)})
				 \prod_{l=1}^{k_j} (r_j s_{j,l})^{2\alpha_{j,l} + p_{j,l} + 1}
				 \prod_{j=1}^\ell  r_j^{k_j-1} \dif r_j \dif S_j \\
		&= \frac{2^n(n+m)!}{m!}
			\int_{\R_+^\ell} \frac{a(r_1, \dots, r_\ell) \prod_{j=1}^\ell r_j^{2|\alpha_{(j)}| + |p_{(j)}| + k_j - 1}} {(1+ r_1^2 +\dots + r_n^2)^{n+m+1}} \\
		&\quad\times
			\bigg(\int_{\prod_{j=1}^\ell S^{k_j-1}} \prod_{j=1}^\ell b_j(r_1, \dots, r_\ell, s_{(j)}) 
				\prod_{l=1}^{k_j} s_{j,l}^{2\alpha_{j,l} + p_{j,l} + 1} \dif S_j 
			\bigg) \dif r  \\
		&= \frac{2^{n-\ell}(n+m)!}{m!}
			\int_{\R_+^\ell} \frac{a(\sqrt{r_1}, \dots, \sqrt{r_\ell}) \prod_{j=1}^\ell r_j^{|\alpha_{(j)}| + \frac{1}{2}|p_{(j)}| + k_j - 1}} {(1+ r_1 +\dots + r_n)^{n+m+1}} \\
		&\quad\times
			\bigg(\int_{\prod_{j=1}^\ell S^{k_j-1}} \prod_{j=1}^\ell b_j(\sqrt{r_1}, \dots, \sqrt{r_\ell}, s_{(j)}) 
				\prod_{l=1}^{k_j} s_{j,l}^{2\alpha_{j,l} + p_{j,l} + 1} \dif S_j 
			\bigg) \dif r	\\
\intertext{We now parametrize spheres by balls to obtain}			
		&= \frac{2^{n-\ell}(n+m)!}{m!}
			\int_{\R_+^\ell} \frac{a(\sqrt{r_1}, \dots, \sqrt{r_\ell}) \prod_{j=1}^\ell r_j^{|\alpha_{(j)}| + 	\frac{1}{2}|p_{(j)}| + k_j - 1}} {(1+ r_1 +\dots + r_n)^{n+m+1}} \\
		&\quad\times
			\bigg(\int_{\prod_{j=1}^\ell \bfB_+^{k_j-1}}
				\bigg(1 - \sum_{l=1}^{k_j-1} s_{j,l}^2 \bigg)^{\alpha_{j,l} + \frac{1}{2}p_{j,l}}  \\
		&\qquad\times					
				\prod_{j=1}^\ell b_j(\sqrt{r_1}, \dots, \sqrt{r_\ell}, s_{(j)}) 
					\prod_{l=1}^{k_j-1} s_{j,l}^{2\alpha_{j,l} + p_{j,l} + 1}\dif s_{j,l} 
			\bigg) \dif r	
\intertext{And now balls by simplices}
		&= \frac{(n+m)!}{m!}
			\int_{\R_+^\ell} \frac{a(\sqrt{r_1}, \dots, \sqrt{r_\ell}) \prod_{j=1}^\ell r_j^{|\alpha_{(j)}| + 	\frac{1}{2}|p_{(j)}| + k_j - 1}} {(1+ r_1 +\dots + r_n)^{n+m+1}} \\
		&\quad\times
			\bigg(\int_{\prod_{j=1}^\ell \Delta_{k_j-1}}
				\bigg(1 - \sum_{l=1}^{k_j-1} s_{j,l} \bigg)^{\alpha_{j,l} + \frac{1}{2}p_{j,l}}  \\
		&\qquad\times
				\prod_{j=1}^\ell b_j(\sqrt{r_1}, \dots, \sqrt{r_\ell}, \sqrt{s_{(j)}}) 
					\prod_{l=1}^{k_j-1} s_{j,l}^{\alpha_{j,l} + \frac{1}{2}p_{j,l}}\dif s_{j,l} 
			\bigg) \dif r.	
	\end{align*}
	Finally, we use that
	\[
		\langle z^{\alpha + p}, z^{\alpha + p}\rangle_m 
		= \frac{(\alpha+p)!(m-|\alpha|)!}{m!} 
		= \frac{(m-|\alpha|)!\prod_{j=1}^\ell\prod_{l=1}^{k_j}(\alpha_{j,l}+p_{j,l})!}{m!},
	\]
	because $|p| = 0$, to obtain the result.
\end{proof}

The following is a particular case.

\begin{corollary}\label{cor:Toeplitz_extended_pj0}
	For $\psi$ as above, we assume that $|p_{(j)}| = 0$ for every $j = 1, \dots, \ell$. Then, the function $\gamma_{\psi, m} : J_n(m) \rightarrow \C$ from Theorem~\ref{thm:Toeplitz_extended} satisfies the following.
	\begin{align*}
		&\gamma_{\psi,m}(\alpha) = \\
		&= \frac{(n+m)!}{(m-|\alpha|)! \prod_{j=1}^\ell(k_j-1+|\alpha_{(j)}|)!}
			\int_{\R^\ell_+} \frac{a(\sqrt{r_1},\dots ,\sqrt{r_\ell})}{(1+r_1+\dots + r_\ell)^{n+m+1}}  \\
		&\quad\times
			\bigg(
				\int_{\prod_{j=1}^\ell \Delta_{k_j-1}} 
					\prod_{j=1}^\ell 
					\frac{(k_j-1+|\alpha_{(j)}|)!}{\prod_{l=1}^{k_j}(\alpha_{j,l}+p_{j,l})!}  
				\bigg(1-\bigg(\sum_{l=1}^{k_j-1} s_{j,l}\bigg)\bigg)^{\alpha_{j,k_j}+\frac{1}{2}p_{j,k_j}}  \\ 
		&\qquad\times 
				b_j(\sqrt{r_{1}},\dots , \sqrt{r_{\ell}},\sqrt{s_{(j)}}) 		
				\prod_{l=1}^{k_j-1} s_{j,l}^{\alpha_{j,l}+\frac{1}{2}p_{j,l}}  \dif s_{j,l}
			\bigg) 
				\prod_{j=1}^\ell r_j^{|\alpha_{(j)}|+k_j-1} \dif r_j.
	\end{align*}	
\end{corollary}

As a consequence we also obtain the following result

\begin{corollary}\label{cor:Toeplitz-commute-extended}
	Let $\psi = a \prod_{j=1}^\ell b_j t_{(j)}^{p_{(j)}}$ be a symbol as in Corollary~\ref{cor:Toeplitz_extended_pj0}. Then, the Toeplitz operators $T_a$, $T_{b_j t_{(j)}^{p_{(j)}}}$ pairwise commute on every weighted Bergman space $\mA^2_m(\PnC)$.
\end{corollary}

\begin{remark}\label{rmk:TabnotTaTb_extended}
	As in the case of Remark~\ref{rmk:TabnotTaTb} it is not difficult to prove that in the previous result and for suitable symbols we have $T_{a\psi_j} \not= T_a T_{\psi_j}$.
\end{remark}

\section{Extended pseudo-homogeneous symbols on the unit ball}\label{sec:extended_unit_ball}
The extended pseudo-homogeneous symbols introduced in Section~\ref{sec:extended-symbols} can be easily considered on the unit ball $\B^n$. As before, we fix a partition $k \in \N^n$. The notation for coordinates is the same as in the previous sections.

\begin{defn}\label{defn:extended-unit-ball}
	A symbol $\psi \in L^\infty(\B^n)$ is called extended $k$-pseudo-homogeneous if it has the form
	\[  
		\psi(z) = b(r_1, \dots, r_\ell, s_{(1)}, \dots ,s_{(\ell)}) t^p
		= b(r_1, \dots , r_\ell, s_{(1)}, \dots, s_{(\ell)}) \prod_{j=1}^\ell t_{(j)}^{p_{(j)}},
	\]
	where $b \in L^\infty(\bfB_+^\ell \times S_+^{k_1-1} \times \dots \times S_+^{k_\ell-1})$ and $p \in \Z^n$ satisfies $|p| = 0$.
\end{defn}

As before, a particular case that is important to consider is given by
\begin{equation}\label{eq:extended-symbol-ball}
	\psi(z)=a(r_1, \dots, r_\ell) \prod_{j=1}^\ell b_j(r_{1},\dots, r_{\ell}, s_{(j)}) t_{(j)}^{p_{(j)}}.
\end{equation}
Using the similar computations to those from Theorem~\ref{thm:Toeplitz_extended} we obtain the following result. We leave the details to the reader.

\begin{thm}\label{thm:Toeplitz_extended_unit_ball}
	Let $\psi \in L^\infty(\B^n)$ be an extended $k$-pseudo-homogeneous symbol whose expression is given as in \eqref{eq:extended-symbol-ball}. Then, the Toeplitz operators $T_\psi$ action on the weighted Bergman space $\mA^2_\lambda(\B^n)$ satisfies
	\[
		T_\psi(z^\alpha) = 
		\begin{cases}
			\gamma_{\psi,\lambda}(\alpha) z^{\alpha + p}, &\text{ if } \alpha + p \geq 0 \\
			0, &\text{ if } \alpha + p \not\geq 0
		\end{cases},
	\]
	for every $\alpha \geq 0$, where 
	\begin{align*}
		&\gamma_{\psi,\lambda}(\alpha) = \\
		&= \frac{\Gamma(n + |\alpha| + \lambda + 1)}{\Gamma(\lambda+1)\prod_{j=1}^\ell 
			(k_j - 1 + |\alpha_{(j)}| + \frac{1}{2}|p_{(j)}|)!} 
			\int_{\Delta_\ell} a(\sqrt{r_1}, \dots, \sqrt{r_\ell})
				\bigg(1 - \sum_{j=1}^\ell r_j \bigg)^\lambda \\
		&\quad\times 
			\bigg(\int_{\prod_{j=1}^\ell \Delta_{k_j-1}}
				\prod_{j=1}^\ell \frac{(k_j - 1 + |\alpha_{(j)}| + \frac{1}{2}|p_{(j)}|)!}{\prod_{l=1}^{k_j} (\alpha_{j,l} + p_{j,l})!}
					\bigg(1-\bigg(\sum_{l=1}^{k_j-1} s_{j,l}\bigg)\bigg)^{\alpha_{j,k_j}+\frac{1}{2}p_{j,k_j}}  \\
		&\qquad\times 
			b_j(\sqrt{r_{1}},\dots , \sqrt{r_{\ell}},\sqrt{s_{(j)}}) 		
				\prod_{l=1}^{k_j-1} s_{j,l}^{\alpha_{j,l}+\frac{1}{2}p_{j,l}}  \dif s_{j,l}
			\bigg) 
					\prod_{j=1}^\ell r_j^{|\alpha_{(j)}|+\frac{1}{2}|p_{(j)}|+k_j-1} \dif r_j.	
	\end{align*}
\end{thm}

As a consequence, we have the following commutativity results.

\begin{corollary}\label{cor:Toeplitz_extended_pj0_unit_ball}
	For $\psi$ as above, we assume that $|p_{(j)}| = 0$ for every $j = 1, \dots, l$. Then, the function $\gamma_{\psi,\lambda} : \N^n \rightarrow \C$ from Theorem~\ref{thm:Toeplitz_extended_unit_ball} satisfies
	\begin{align*}
	&\gamma_{\psi,\lambda}(\alpha) = \\
	&= \frac{\Gamma(n + |\alpha| + \lambda + 1)}{\Gamma(\lambda+1)\prod_{j=1}^\ell 
			(k_j - 1 + |\alpha_{(j)}|)!} 
			\int_{\Delta_\ell} a(\sqrt{r_1}, \dots, \sqrt{r_\ell})
				\bigg(1 - \sum_{j=1}^\ell r_j \bigg)^\lambda \\
	&\quad\times 
			\bigg(\int_{\prod_{j=1}^\ell \Delta_{k_j-1}}
				\prod_{j=1}^\ell \frac{(k_j - 1 + |\alpha_{(j)}|)!}{\prod_{l=1}^{k_j} (\alpha_{j,l} + p_{j,l})!}
	\bigg(1-\bigg(\sum_{l=1}^{k_j-1} s_{j,l}\bigg)\bigg)^{\alpha_{j,k_j}+\frac{1}{2}p_{j,k_j}}  \\
	&\qquad\times 
			b_j(\sqrt{r_{1}},\dots , \sqrt{r_{\ell}},\sqrt{s_{(j)}}) 		
			\prod_{l=1}^{k_j-1} s_{j,l}^{\alpha_{j,l}+\frac{1}{2}p_{j,l}}  \dif s_{j,l}
			\bigg) 
			\prod_{j=1}^\ell r_j^{|\alpha_{(j)}|+k_j-1} \dif r_j.	
	\end{align*}	
\end{corollary}

\begin{corollary}\label{cor:Toeplitz-commute-extended-unitball}
	Let $\psi = a \prod_{j=1}^\ell b_j t_{(j)}^{p_{(j)}}$ be a symbol as in Corollary~\ref{cor:Toeplitz_extended_pj0_unit_ball}. Then, the Toeplitz operators $T_a, T_{b_j t_{(j)}^{p_{(j)}}}$ pairwise commute on every weighted Bergman space $\mA^2_\lambda(\B^n)$.
\end{corollary}

\begin{remark}\label{rmk:TabnotTaTb_extended_unitball}
	In this case we have $T_{a\psi_j} \not= T_a T_{\psi_j}$ as in the previous corresponding remarks.
\end{remark}


\begin{thebibliography}{XX}
	\bibitem{Audin}  Audin, Mich\`ele: \textit{Torus actions on symplectic manifolds}. Second revised edition. Progress in Mathematics, 93. Birkh\"auser Verlag, Basel, 2004.

	\bibitem{Vasilevski2015}  Garc\'{\i}a, A.; Vasilevski, N.: \textit{Toeplitz Operators on the Weighted Bergman Space over the Two-Dimensional Unit Ball}, Journal of Function Spaces (2015), DOI:10.1155/2015/306168, p.10. 

	\bibitem{GQV-disk}  Grudsky, S.; Quiroga-Barranco, R.; Vasilevski, N.: \textit{Commutative C∗-algebras of Toeplitz operators and quantization on the unit disk}. J. Funct. Anal. 234 (2006), no. 1, 1--44.

	\bibitem{MR-SN-RO}  Morales-Ramos, Miguel A.; Sanchez-Nungaray, Armando; Ramirez-Ortega, Josue: \textit{Toeplitz operators with quasi-separately radial symbols on the complex projective space}, Bol. Soc. Mat. Mex. (3) \textbf{22} (2016), no. 1, 213--227. 
	
	\bibitem{Ortega-Ratiu}  Ortega, Juan-Pablo; Ratiu, Tudor S.: \textit{Momentum maps and Hamiltonian reduction}. Progress in Mathematics, 222. Birkh\"auser Boston, Inc., Boston, MA, 2004.
	
	\bibitem{QS-Projective} Quiroga-Barranco, Raul; Sanchez-Nungaray, Armando: \textit{Commutative $C^*$-algebras of Toeplitz operators on complex projective spaces}, Integral Equations and Operator Theory \textbf{71} (2011), no. 2, 225-243.
	
	\bibitem{QS-Quasi-Projective} Quiroga-Barranco, Raul; Sanchez-Nungaray, Armando: \textit{Toeplitz operators with quasi-radial quasi-homogeneous symbols and bundles of Lagrangian frames}, Journal of Operator Theory \textbf{71} (2014), no. 1, 199--222.
	
	\bibitem{QV-2007} Quiroga-Barranco, Raul; Vasilevski, Nikolai: \textit{Commutative C*-Algebras of Toeplitz Operators on the Unit Ball, I. Bargmann-Type Transforms and Spectral Representations of Toeplitz Operators}, Integral Equations and Operator Theory \textbf{59} (2007), no. 3, 379--419.
	
	\bibitem{QV-2008} Quiroga-Barranco, Raul; Vasilevski, Nikolai: \textit{Commutative C*-Algebras of Toeplitz Operators on the Unit Ball, II. Geometry of the Level Sets of Symbols}, Integral Equations and Operator Theory \textbf{60} (2008), no. 1, 89--132.
		
	\bibitem{Vasilevski2010} Vasilevski, Nikolai: \textit{Quasi-Radial Quasi-Homogeneous Symbols and Commutative Banach Algebras of Toeplitz Operators}, Integral Equations and Operator Theory \textbf{66} (2010), no. 1, 141--152.
	
	\bibitem{Vasilevski-Pseudo} Vasilevski, Nikolai: \textit{On Toeplitz operators with quasi-radial and pseudo-homogeneous symbols}, preprint.
	
\end{thebibliography}
\end{document}